\documentclass{amsart}

\usepackage{mathrsfs}
\usepackage{stmaryrd}
\usepackage{amscd}
\usepackage{amssymb}
\usepackage[alphabetic]{amsrefs}
\usepackage[all]{xy}
\usepackage{mathtools}

\title[generic cohomology of local and global Shimura varieties]
{On the generic part of the cohomology of local and global Shimura varieties}
\author{Teruhisa Koshikawa}
\address{Research Institute for Mathematical Sciences, Kyoto University}
\email{teruhisa@kurims.kyoto-u.ac.jp}

\theoremstyle{plain}
\newtheorem{thm}{Theorem}[section]
\newtheorem{lem}[thm]{Lemma}
\newtheorem{prop}[thm]{Proposition}
\newtheorem{cor}[thm]{Corollary}

\theoremstyle{definition}
\newtheorem{rem}[thm]{Remark}

\newtheorem{conj}[thm]{Conjecture}

\newcommand\fm{\mathfrak m}
\newcommand\fp{\mathfrak p}

\newcommand\bA{\mathbb A}

\newcommand\bX{\mathbb X}

\newcommand\bC{\mathbf C}
\newcommand\bD{\mathbf D}
\newcommand\bF{\mathbf F}
\newcommand\bQ{\mathbf Q}
\newcommand\bZ{\mathbf Z}

\newcommand\cD{\mathcal D}
\newcommand\cE{\mathcal E}
\newcommand\cF{\mathcal F}
\newcommand\cG{\mathcal G}
\newcommand\cM{\mathcal M}

\newcommand\cO{\mathcal O}
\newcommand\cS{\mathcal S}

\newcommand\cZ{\mathcal Z}

\newcommand\Fl{\mathscr F \ell}

\newcommand{\ov}{\overline}

\DeclareMathOperator{\GL}{GL}		
\DeclareMathOperator{\Frob}{Frob}
\DeclareMathOperator{\End}{End}

\DeclareMathOperator{\Hom}{Hom}
\DeclareMathOperator{\Rep}{Rep}
\DeclareMathOperator{\id}{id}

\DeclareMathOperator{\Res}{Res}
\DeclareMathOperator{\cInd}{c-Ind}
\DeclareMathOperator{\Bun}{Bun}
\DeclareMathOperator{\Spa}{Spa}
\DeclareMathOperator{\Spd}{Spd}

\DeclareMathOperator{\Ext}{Ext}
\DeclareMathOperator{\semi}{ss}
\DeclareMathOperator{\op}{op}
\DeclareMathOperator{\lis}{lis}

\DeclareMathOperator{\spec}{spec}
\DeclareMathOperator{\Ig}{Ig}
\DeclareMathOperator{\pr}{pr}
\DeclareMathOperator{\unip}{unip}

\newcommand{\et}{\mathrm{\acute{e}t}}	

\begin{document}

\begin{abstract}
Using the work of Fargues-Scholze, we prove a vanishing theorem for the generic unramified part of the cohomology of local Shimura varieties of general linear groups. This gives an alternative approach to vanishing results of Caraiani-Scholze for the cohomology of unitary Shimura varieties. 
\end{abstract}

\maketitle

\section{Introduction}
\subsection{Local vanishing}
Let $F$ be a finite extension of $\bQ_p$ for some $p$, with the residue field isomorphic to $\bF_{q}$. Take a local Shimura datum $(G, b, \mu)$ over $F$ with $G$ being a product of general linear groups $\GL_{n_i}$ over $F$ indexed by a finite set $I$. 
Let $E\subset \overline{F}$ denote the reflex field and put $C:=\widehat{\overline{E}}$. 
We will work with a hyperspecial open compact subgroup $K:=\prod_{i\in I} \GL_{n_i} (\cO_{F})\subset G (F)$, and  
let $\cM_{(G, b, \mu), K}$ denote the corresponding local Shimura variety of level $K$ \cite{Scholze:Berkeley}*{24.1.3}. Consider the compactly supported cohomology $R\Gamma_c (\cM_{(G, b, \mu), K}, \bZ_{\ell})$ for $\ell \neq p$ (taken after base change to $C$), equipped with \emph{left} action of $J_b (F)$, and \emph{right} action of the Hecke algebra $H_K :=\bZ_{\ell}[K \backslash G(F) / K]$. (The Weil group $W_E$ also acts but this will not be relevant to us.)

Suppose $\fm \subset H_K$ is the maximal ideal corresponding to a character $\lambda_{\fm}\colon H_K \to \overline{\bF}_{\ell}$. We have the associated unramified $L$-parameter
\[
\rho_{\fm}\colon \Frob_F^{\bZ} \to \prod_{i\in I} \GL_{n_i} (\overline{\bF}_{\ell}); 
\]
this is just determined by a conjugacy class of semisimple elements. 
Let us say that $\rho_{\fm}$ is \emph{generic} if, for every $i\in I$, the eigenvalues $\alpha_1, \dots ,\alpha_{n_i}$ of $\rho_{\fm} (\Frob_{F})$ regarded as an element of $\GL_{n_i}(\overline{\bF}_{\ell})$ satisfy $\alpha_{j'}/\alpha_{j} \neq q$ for all $j\neq j'$. 

\begin{thm}\label{local vanishing}
If $\rho_{\fm}$ is \emph{generic} and $J_b$ is \emph{not quasi-split}, then the localized cohomology
\[
H^i_c (\cM_{(G, b, \mu), K}, \bZ_{\ell})_{\fm}
\]
vanishes for every integer $i$. 
\end{thm}

The work of Fargues-Scholze \cite{FS} is fundamental in our proof. In some sense, they prove that ``Jacquet-Langlands transfer'' appears in the cohomology of local Shimura varieties even for mod $\ell$ coefficients. We also use the results of \cite{HKW} and \cite{MS:lift} to relate it to the actual Jacquet-Langlands correspondence in characteristic $0$.  
The genericity assumption implies that $\lambda_{\fm}$ is not in the image of ``Jacquet-Langlands transfer''. 
(The underlying idea is loosely related to some arguments in \cites{CS, CS:noncompact}, e.g., \cite{CS}*{5.4.3}.)
It is possible to prove a variant for not necessarily generic representations by putting more restriction on $J_b$ instead; this is related to the work of Boyer on the Harris-Taylor Shimura variety \cite{Boyer}, and will be discussed elsewhere. 

\subsection{Application to unitary Shimura varieties}

We can use Theorem \ref{local vanishing} to give an alternative argument for vanishing theorems of \cites{CS, CS:noncompact}. To be precise, our argument only uses their results on the geometry of the Hodge-Tate period map. 

Let $F$ be a CM field. 
Let $(B, *, V, (\cdot, \cdot))$ be a PEL datum of type A with $B$ being a central simple $F$-algebra, and let $(G, X)$ denote the associated Shimura datum.
Take a prime $p$ that \emph{completely splits} over $F$. 
In particular, choosing suitable places $v_1, \dots , v_a$ of $F$ dividing $p$, we have an isomorphism
\[
G_{\bQ_p} \cong \mathbb{G}_m^{\times} \times \prod_{i} G_{v_i}, 
\]
where $G_{v_i}$ is a product of general linear groups over $F_{v_i}(\cong \bQ_p)$. 
Take a hyperspecial open compact subgroup $K_p=\bZ_p^{\times}\times \prod K_{v_i}$ of $G(\bQ_p)$ . 
The Hecke algebra at $p$
\[
H_{K_p}:=\bZ_{\ell}[K_p \backslash G(\bQ_p) /K_p]
\]
acts via \emph{right} on the cohomology of Shimura variety
\[
H^i (S_K, \overline{\bF}_{\ell}), \quad H^i_c (S_K, \overline{\bF}_{\ell})
\]
for sufficiently small open compact subgroup $K=K_p K^p \subset G(\bQ_p)\times G(\bA^{p}_f)$. 
(We are considering the right action, but it is common in the literature to consider the left action obtained by composing the right action with an involution of $H_{K_p}$.)
Set $d:=\dim S_K$. 

\begin{conj}\label{conj}
Let $\lambda_{\fm_p}\colon H_K\to \overline{\bF}_{\ell}$ denote the character corresponding to a maximal ideal  $\fm_p \subset H_{K_p}$ such that $\rho_{\fm_p}$ is \emph{generic}. 
If $H^i (S_K, \overline{\bF}_{\ell})_{\fm_p}\neq 0$, then $i\geq d$. 
If $H^i_c (S_K, \overline{\bF}_{\ell})_{\fm_p}\neq 0$, then $i\leq d$. 
\end{conj}

Let us restrict ourselves to the settings of \cites{CS, CS:noncompact}:

\begin{thm}\label{CS}
Assume $B=F$, $V=F^{2n}$, and $G$ is a quasi-split similitude unitary group. 
Then, Conjecture \ref{conj} holds true. 
\end{thm}

\begin{thm}\label{compact}
Assume $G$ is anisotropic modulo center. 
Then, Conjecture \ref{conj} holds true: if $H^i (S_K, \overline{\bF}_{\ell})_{\fm_p}\neq 0$, then $i= d$. 
\end{thm}

Theorem \ref{CS} (resp. Theorem \ref{compact}) is proved in \cite{CS:noncompact}*{Theorem 1.1} (resp. \cite{CS}) after localizing at a maximal ideal of the global Hecke algebra under some assumptions, and their version of Theorem \ref{CS} plays a very important role in the application to (potential) automorphy \cite{ten authors}. 
Note however that our statement only uses the local Hecke algebra at $p$. Our method uses neither trace formula nor detailed analysis of the cohomology of Igusa varieties and their boundaries, while we still crucially relies on the semiperversity result of Caraiani-Scholze \cite{CS:noncompact}*{Theorem 4.6.1}. 
(This is the only reason we work in the setting of \cite{CS:noncompact}. Once their results on the geometry of Igusa varieties are generalized, Conjecture \ref{conj} can be proved by our method.)

Let us describe the strategy of our proof. 
We start with the following form of Mantovan's formula (see Theorem \ref{Mantovan}): $R\Gamma (S_{K, \ov{\bQ}}, \overline{\bF}_{\ell})_{\fm_p}$ admits a $H_{K_p}$-equivariant filtration (in the derived sense) whose graded pieces are
\begin{multline*}
R\Gamma (\Ig^b, \ov{\bF}_{\ell})^{\op}\otimes^L_{C_c (J_b (\bQ_p))} R\Gamma_c (\cM_{(G, b, \mu), \infty}, \ov{\bF}_{\ell}(d_b))_{\fm_p} [2d_b] \cong \\
(R\Gamma_c (\Ig^b, \ov{\bF}_{\ell}(d_b))^{*})^{\op}\otimes^L_{C_c (J_b (\bQ_p))} R\Gamma_c (\cM_{(G, b, \mu), \infty}, \ov{\bF}_{\ell}(d_b))_{\fm_p}
\end{multline*}
for $b\in B(G_{\bQ_p}, \mu^{-1})$, where $\Ig^b$ denotes the corresponding Igusa variety, $d_b=\dim \Ig^b$,  and $(-)^{*}$ denotes the smooth dual. 

The first step is to use Theorem \ref{local vanishing}. 

\begin{prop}\label{Key 1}
If $b$ is not ordinary, 
\[
R\Gamma (\Ig^b, \overline{\bF}_{\ell})^{\op} \otimes^L_{C_c (J_b (\bQ_p))} R\Gamma_c (\cM_{(G_{\bQ_p}, b, \mu), K_p}, \overline{\bF}_{\ell}(d_b))_{\fm_p}
\cong 
0.
\]
\end{prop}

\begin{proof}
To apply Theorem \ref{local vanishing}, use that $J_b$ is quasi-split if and only if $b$ is ordinary in the current setting; cf. proofs of \cite{CS}*{5.5.4, 5.5.5}.  Also note that the twist $\ov{\bF}_{\ell}(d_b)$ does not affect the Hecke action. 
\end{proof}

It remains to work with the ordinary term; let $b_0$ denote the ordinary element. 
For this, we recall the Hodge-Tate period map, which is actually hidden in Mantovan's formula. 
Let $\cS_{K^p}$ denote the perfectoid Shimura variety of level $K^p$. 
Let $\cS_{K^p}^{\circ}$ denote the good reduction locus, which lives over the adic generic fiber of the formal completion of the integral model of level $K$. 
Let $\Fl$ denote the flag variety of $G_{\bQ_p}$ associated with $\mu$ (or $\mu^{-1}$, depending on the sing convention). 
We have the $G(\bQ_p)$-equivariant Hodge-Tate period maps \cite{CS}*{2.1.3}:
\[
\pi_{HT}\colon \cS_{K^p} \to \Fl, \quad
\pi_{HT}^{\circ}\colon \cS_{K^p}^{\circ} \to \Fl
\]
as diamonds; the second is just the restriction of the first. 
We also consider their quotients by $K_p$, and use the same notation. 
We now want to show that 
\[
R\Gamma_c ([\Fl (\bQ_p)/\underline{K_p}], i^{b_0*}R\pi^{\circ}_{HT*}\ov{\bF}_{\ell})_{\fm_p}
\]
sits in degree $\geq d$; here, $\Fl (\bQ_p)$ is the ordinary locus and 
\[
i^{b_0}\colon [\Fl(\bQ_p)/ \underline{K_p}] \hookrightarrow [\Fl_C / \underline{K_p}]
\]
is a closed immersion of v-stacks. 
The key result from \cites{CS:noncompact} is

\begin{thm}[\cite{CS:noncompact}]
Let $C$ be any complete algebraically closed nonarchimedean extension of $\bQ_p$. 
Then, $R\pi_{HT*}\ov{\bF}_{\ell}$ belongs to $^{p}D^{\geq d} (\Fl_C, \ov{\bF}_{\ell})$; the precise claim involves integral models and nearby cycles.  
\end{thm}

We can deduce from this that 
\[
R\Gamma_c ([\Fl (\bQ_p)/ \underline{K_p}], Ri^{b_0!}R\pi^{\circ}_{HT*}\ov{\bF}_{\ell})
\]
lives in degree $\geq d$. 
The final step is to kill the contribution from the non-ordinary loci after localizing at $\fm_p$: 

\begin{prop}\label{Key 2}
A natural map
\[
R\Gamma_c ([\Fl (\bQ_p)/\underline{K_p}], Ri^{b_0!}R\pi^{\circ}_{HT*}\ov{\bF}_{\ell})_{\fm_p} 
\to
R\Gamma_c ([\Fl (\bQ_p)/\underline{K_p}], i^{b_0*}R\pi^{\circ}_{HT*}\ov{\bF}_{\ell})_{\fm_p}. 
\]
is an isomorphism. 
\end{prop}

This is proved again using the machinery of \cites{FS}. 
There is a $H_{K_p}$-equivariant filtration on $Ri^{b_0!} R\pi^{\circ}_{HT*}\ov{\bF}_{\ell}$ whose graded pieces are 
\[
Ri^{b_0!} i^b_{!}i^{b \, *} R\pi^{\circ}_{HT*}\ov{\bF}_{\ell}
\]
where $i^b\colon [\Fl^b /\underline{K_p}]\hookrightarrow [\Fl  /\underline{K_p}]$ is a locally closed embedding. 
The piece from $b_0$ is exactly $i^{b_0*} R\pi^{\circ}_{HT*}\ov{\bF}_{\ell}$. 
Now, it suffices to control the composite $Ri^{b_0 !} i^b_!$ for a non-odrinary $b$. 
This can be reduced to an analogous problem in the context of the moduli stack $\Bun_{G_{\bQ_p}}$ of $G_{\bQ_p}$-bundles on the Fargues-Fontaine curve studied in \cite{FS}: one can show that $Ri^{b_0 !} i^b_! i^{b*}R\pi^{\circ}_{HT *}$ comes from the corresponding object from $\Bun_{G_{\bQ_p}}$ as $i^{b*}R\pi^{\circ}_{HT *}$ comes from the strata $\Bun_{G_{\bQ_p}}^{b}$ corresponding to $b$. 
Thus, we can write the contribution from $b$ in the form of
\[
(Ri^{b_0 !}i^b_! (-))^{\op}
\otimes^L_{C_c (J_{b_0} (\bQ_p))}
R\Gamma_c (\cM_{(G_{\bQ_p}, b_0,\mu), K_p}, \ov{\bF}_{\ell}(d))_{\fm^p}
\]
for $i^b\colon \Bun_{G_{\bQ_p}}^b \hookrightarrow \Bun_{G_{\bQ_p}}$. 
To show that such contribution vanishes, we look at actions of excursion operators introduced in \cite{FS}: the action on the right term is ``generic'' by the assumption on $\fm_p$, while the action on the left term cannot be ``generic'' as it comes from the non-split group $J_b$.

\subsection*{Acknowledgements}
This work was supported by JSPS KAKENHI Grant Number 20K14284.

\section{Mod $\ell$ representations of inner forms of general linear groups}
Let $F$ be a finite extension of $\bQ_p$ with residue field $\bF_q$. Let $G$ be an inner form of $\GL_n$ over $F$. 
Let $\ell$ be a prime different from the residue characteristic of $F$ and fix $q^{1/2} \in \overline{\bF}_{\ell}$. 
Let $\Rep_{\overline{\bF}_{\ell}} (G)$ denote the category of smooth $\overline{\bF}_{\ell}$-representation of $G(F)$. 
We will recall some results on $\Rep_{\overline{\bF}_{\ell}} (G)$. 

First recall that an irreducible representation $\pi$ of $G(F)$ is called \emph{cuspidal} (resp. \emph{supercuspidal}) if $\pi$ is not a quotient (resp. subquotient) of any properly parabolically induced representation. 
Two notions are different in general. 
The pair $(M, \sigma)$ consisting of a Levi subgroup of $G$ and a supercuspidal irreducible representation $\sigma$ of $M(F)$ is the supercuspidal support of $\pi$ if $\pi$ is a subquotient of the \emph{normalized} parabolic induction $i^G_M \sigma$; such $(M, \sigma)$ always exists. (For the current $G$, the pair $(M, \sigma)$ is known to be unique up to $G$-conjugacy \cite{MS:supercuspidal support}*{Th\'eor\`eme A}.)

We need the following result on supercuspidal representations: 

\begin{thm}[\cite{MS:lift}*{3.27, 3.28}]\label{lift}
Any supercuspidal irreducible $\bF_{\ell}$-representation admits a lift to a supercuspidal irreducible $\bQ_{\ell}$-representation. 
\end{thm}

\section{A lemma on $L$-parameters}
Using the results explained in the previous section, we prove a lemma on $L$-parameters. 

Let $F$ be a finite extension of $\bQ_p$ with residue field $\bF_q$. 
Let $G$ be a product of general linear groups over $F$. 
Fix $b \in B(G)$, and let $J_b$ denote the associated algebraic group over $F$ as usual. 

Fix $\ell$ different from $p$ and $q^{1/2}\in \overline{\bZ}_{\ell}$.  
Let $L$ be an algebraically closed field over $\overline{\bZ}_{\ell}$. 
For any irreducible smooth $L$-representation $\pi$ of $J_b (F)$, Fargues-Scholze \cite{FS}*{IX.4.1, IX.7.1} constructed the (semisimple) $L$-parameter well-defined up to $\widehat{J}_b(L)$-conjugacy:
\[
\varphi_{\pi}\colon W_F \to \widehat{J}_b (L) \times W_F \hookrightarrow \widehat{G}(L) \times W_F,  
\]
where the embedding is twisted as in \cite{FS}*{IX.7.1}. 
(Note that $\pi$ is known to be Schur-irreducible.)

Let us recall some properties of $L$-parameters. 
\begin{enumerate}
\item The construction is compatible with products of groups \cite{FS}*{IX.6.2}: if $G=G_1 \times G_2$, $b=(b_1, b_2)$, and $\pi=\pi_1 \boxtimes \pi_2$, then $\varphi_{\pi}=(\varphi_{\pi_1}, \varphi_{\pi_2})$ as homomorphisms to $\widehat{J}_b (L)=\widehat{J}_{b_1} (L)\times \widehat{J}_{b_2} (L)$.  
\item The construction is compatible with the parabolic induction \cite{FS}*{IX.7.3}: if $\pi$ is a subquotient of a \emph{unnormalized} parabolically induced representation $\textnormal{Ind}^{J_b}_M \sigma$, then $\varphi_{\pi}$ is obtained from $\varphi_{\sigma}$ via the \emph{twisted} embedding $\widehat{M}(L)\hookrightarrow \widehat{J}_b(L)$.   
Equivalently, if $\pi$ is a subquotient of a \emph{normalized} parabolically induced representation $i^{J_b}_M \sigma$, then $\varphi_{\pi}$ is obtained from $\varphi_{\sigma}$ via the \emph{untwisted} embedding $\widehat{M}(L)\hookrightarrow \widehat{J}_b(L)$. 
\item The construction is compatible with the mod $\ell$ reduction \cite{FS}*{IX.5.2}: if $\pi$ is an irreducible $\overline{\bF}_{\ell}$-representation and obtained as the mod $\ell$ reduction of an irreducible $\overline{\bQ}_{\ell}$-representation $\widetilde{\pi}$, $\varphi_{\pi}$ is the mod $\ell$ reduction of $\varphi_{\widetilde{\pi}}$. 
\item For an irreducible $\overline{\bQ}_{\ell}$-representation $\pi$, $\varphi_{\pi}$ is the same as the usual semisimplified $L$-parameter (i.e., it is compatible with the one of Harris-Taylor via the Jacquet-Langlands correspondence if we ignore the monodromy operator); this is \cite{HKW}*{1.0.3}. 
\end{enumerate}

Now we can prove the lemma we need:

\begin{lem}\label{non-quasi-split implies non-generic}
Suppose $J_b$ is not quasi-split. 
If $\pi$ is an irreducible smooth $\bF_{\ell}$-representation of $J_b (F)$, then $\varphi_{\pi}$ is not a generic unramified $L$-parameter. 
\end{lem}

\begin{proof}
Write $G=\prod_i \GL_{n_i}$, $b=(b_i)$, $J_b =J_{b_i}$, and $\pi=\boxtimes_i \pi_i$. 
For some $i$, $J_{b_i}$ is not quasi-split by the assumption. 
By the definition of genericity and properties of $L$-parameters, we may assume that $G=\GL_n$. 

Assume $\varphi_{\pi}$ is a generic unramified $L$-parameter. 
Take the supercuspidal support $(M, \sigma)$ of $\pi$, and a lift $\widetilde{\sigma}$ of $\sigma$ using Theorem \ref{lift}. 
The $L$-parameter
\[
\varphi_{\widetilde{\sigma}}\colon W_F \to \widehat{M}(\overline{\bQ}_{\ell}) \hookrightarrow \GL_n(\overline{\bQ}_{\ell})
\]
with the untwisted embedding, factors through $\GL_n(\overline{\bZ}_{\ell})$, and its reduction to $\GL_n(\overline{\bF}_{\ell})$ recovers $\varphi_{\pi}$. 
By the genericity, we have a decomposition $\varphi_{\widetilde{\sigma}}=\oplus_{i=1}^n \chi_i$ into characters $\chi_1, \dots, \chi_n$ and $\chi_i / \chi_j$ is not the cyclotomic character for any $i\neq j$; cf. \cite{CS}*{6.2.2} and \cite{CS:noncompact}*{proof of 5.1.3}. 
Since $J_b$ is not quasi-split, $M$ has a factor of the form of $\GL_m(D)$ for some central division $F$-algebra $D$ of dimension $>1$ and $m\geq 1$. Thus, no such representation $\widetilde{\sigma}$ exists. 
\end{proof}

\section{Proof of Theorem \ref{local vanishing}}
We shall use Hecke operators and excursion operators. 
Let us first recall the description of the cohomology of local Shimura varieties from \cite{FS}*{IX.3}. 
Let $r_\mu$ denote the representation of $\widehat{G}$ whose highest weight is conjugate to the cocharacter corresponding to $\mu$, and set $d=\dim r_{\mu}$. 
Let $i^1$ (resp. $i^b$) denote the immersion $\Bun^1_G \hookrightarrow \Bun_G$ (resp. $\Bun^b_G \hookrightarrow \Bun_G$).   
Then, there is an identification
\[
R\Gamma_c (\cM_{(G,b, \mu), K}, \bZ_{\ell}[q^{1/2}])[d](d/2) \cong 
i^{b \, *}T_{r_{\mu}} (i^1_! \cInd^{G(F)}_{K}\bZ_{\ell}[q^{1/2}])
\]
as representations of $J_b (F)$, and each $H^i_c$ is a finitely generated smooth representation of $J_b (\bQ_p)$ \cite{FS}*{IX.3.1}. 
Moreover, this identification is compatible with \emph{right} actions of the Hecke algebra $H_K$. 

Given $\fm \subset H_K$, it induces 
\[
R\Gamma_c (\cM_{(G,b, \mu), K}, \bZ_{\ell}[q^{1/2}])_{\fm}[d](d/2) \cong 
i^{b \, *}T_{r_{\mu}} (i^1_! (\cInd^{G(F)}_{K}\bZ_{\ell}[q^{1/2}])_{\fm})
\]
as $i^{b*}$, $T_{r_{\mu}}$, and $i^1_!$ commute with colimits.  

Let $\pi$ be an irreducible subquotient of $H^i_c (\cM_{(G,b, \mu), K}, \overline{\bF}_{\ell})_{\fm}$. 
Once we prove that $\varphi_{\pi}=\rho_{\fm}$, we can finish by Lemma \ref{non-quasi-split implies non-generic}. 
To identify $\varphi_{\pi}$, we shall look at excursion operators. Let $\cD=(I, V, \alpha, \beta, (\gamma_i)_{i \in I})$ be an excursion datum \cite{FS}*{VIII.4.2}. 
We need to show that the associated excursion operator $S_{\cD}$ acts on $\pi$ by the scalar
\[
\overline{\bF}_{\ell}\xrightarrow{\alpha} V \xrightarrow{(\varphi_{\fm} (\gamma_i))_{i\in I}} V \xrightarrow{\beta} \overline{\bF}_{\ell}, 
\]
where we are implicitly taking the image of $\gamma_i$ in $\Frob^{\bZ}_F$. 

We will actually work with the spectral Bernstein center $\cZ^{\spec} (G, \bZ_{\ell}[q^{1/2}])$ of $G$ \cite{FS}*{IX.0.2}. 
As $\ell$ is very good for $\widehat{G}$ in their sense, $\cZ^{\spec} (G, \bZ_{\ell}[q^{1/2}])$ acts on $D_{\lis}(\Bun_G, \bZ_{\ell}[q^{1/2}])$ \cite{FS}*{IX.5.2}, and the action is compatible with excursion operators. 
Recall that $\rho_{\fm}$ corresponds to the character $\lambda_{\fm} \colon H_K \to \overline{\bF}_{\ell}$.  
On the other hand, the action of $\cZ^{\spec} (G, \bZ_{\ell}[q^{1/2}])$ on $\cInd^{G(F)}_K \bZ_{\ell}[q^{1/2}]$ induces a natural map
\[
\cZ^{\spec} (G, \bZ_{\ell}[q^{1/2}]) \to \End_{G(F)} (\cInd^{G(F)}_K \bZ_{\ell}[q^{1/2}])= H_K^{\op}. 
\]
(Beware that $\cInd^{G(F)}_K \bZ_{\ell}[q^{1/2}]$ is being regarded as a \emph{right} module of $H_K$.)
Composing it with an involution $KhK \mapsto  Kh^{-1}K$ of $H_K$, we obtain a map
\[
\cZ^{\spec} (G, \bZ_{\ell}[q^{1/2}]) \to H_K;
\]
this is the map compatible with usual $L$-parameters for unramified irreducible representations as $L$-parameters of Fargues-Scholze agree with the usual one. 
Now, $\lambda_{\fm}$ determines the localization
\[
\cZ^{\spec} (G, \bZ_{\ell}[q^{1/2}])\to \cZ^{\spec} (G, \bZ_{\ell}[q^{1/2}])_{\fm}. 
\] 
It suffices to show that 
\begin{quote}
the action of $\cZ^{\spec} (G, \bZ_{\ell}[q^{1/2}])$ on 
$i^{b \, *}T_{r_{\mu}} (i^1_! (\cInd^{G(F)}_{K}\bZ_{\ell}[q^{1/2}])_{\fm})$ factors through $\cZ^{\spec} (G, \bZ_{\ell}[q^{1/2}])_{\fm}$. 
\end{quote}
As the action of $\cZ^{\spec} (G, \bZ_{\ell}[q^{1/2}])$ commute with excursion operators \cite{FS}*{IX.5.2}, we need only observe that 
\begin{quote}
the action of $\cZ^{\spec} (G, \bZ_{\ell}[q^{1/2}])$ on $(\cInd^{G(F)}_{K}\bZ_{\ell}[q^{1/2}])_{\fm}$ factors through $\cZ^{\spec} (G, \bZ_{\ell}[q^{1/2}])_{\fm}$. 
\end{quote}
This completes the proof of Theorem \ref{local vanishing}. 

\section{Shimura varieties of PEL type and the Hodge-Tate period map}
We start the discussion on Shimura varieties. 
Let $(\cO_B, *, \Lambda, (\cdot, \cdot))$ be an integral PEL datum of type A or C unramified at $p$; cf.~\cite{CS}*{4.3}. Let $(G, X)$ denote the associated Shimura datum with the reflex field $E\subset \bC$. 
We write $S_K$ for the Shimura variety of level $K\subset G(\bA_f)$. 
Fix a place $\fp$ of $E$ above $p$ and an embedding $\overline{\bQ}\to \overline{\bQ}_p$ that induces $\fp$. 
Let $\cS_K$ denote the adic space over $E_{\fp}$ associated with $S_{K, E_{\fp}}$. 
Suppose for the moment that $K=K^p K_p$ and $K_p=G(\bZ_p)$. 
The Shimura variety has the canonical integral model $\mathscr{S}_K$ over $\cO_{E, \fp}$. 
Let $\cS^{\circ}_{K}\hookrightarrow \cS_K$ denote the open immersion from the adic generic fiber of the $p$-adic completion $\mathscr{S}_K^{\wedge}$ of $\mathscr{S}_K$; this is the locus of ``good reduction''.  
For deeper $K$, we take the inverse image to define $\cS^{\circ}_{K}\hookrightarrow \cS_K$; the tower $\{\cS_K^{\circ}\}_{K}$ is stable under the Hecke action.   
We also consider the associated diamonds \cites{Scholze:Berkeley, Scholze:diamond}. 

\begin{lem}\label{cohomology of good reduction locus}
Let $C$ be a complete algebraically closed nonarchimedean field over $\bQ_p$. the natural map
\[
R\Gamma (S_{K, \overline{\bQ}}, \overline{\bF}_{\ell}) \to 
R\Gamma (\cS_{K, C}^{\circ}, \ov{\bF}_{\ell})
\]
is an isomorphism for $\ell\neq p$. 
\end{lem}

\begin{proof}
This follows from results of Lan-Stroh \cite{Lan-Stroh}*{5.20} and Huber; see \cite{CS:noncompact}*{2.6.4}. 
\end{proof}

We now define perfectoid Shimura varieties. 
Set
\[
\cS_{K^p}:= \varprojlim_{K_p} \cS_{K^p K_p}^{\diamondsuit}, \quad
\cS_{K^p}^{\circ}:= \varprojlim_{K_p} \cS^{\circ, \diamondsuit}_{K^p K_p} \hookrightarrow \cS_{K^p}; 
\]
both are actually perfectoid spaces after base change to a perfectoid field. 
Let $\Fl$ denote the flag variety of $G(E_{\fp})$ associated with the cocharacter $\mu$ (or, perhaps better, its inverse) determined by $h\in X$; cf. \cite{CS}*{2.1}. 
We regard it as a diamond over $\Spd (E_{\fp})$. 

\begin{thm}[\cite{CS}*{2.1.3}]
There is a $G(\bQ_p)$-equivariant morphism
\[
\pi_{HT}\colon \cS_{K^p}\to \Fl
\]
with a specific construction. 
This is called the Hodge-Tate period map. 
\end{thm}

We write $\pi_{HT}^{\circ}$ for the restriction of $\pi_{HT}$ to $\cS_{K^p}^{\circ}$. The restriction $\pi_{HT}^{\circ}$ is quasicompact and quasiseparated. 

\section{Newton stratifications and Igusa varieties}
It is important to study the Hodge-Tate period map using Newton stratifications. 
There are two types of stratifications. 
We fix a complete algebraically closed nonarchimedean extension $C$ of $\ov{\bQ}_p \supset E_{\fp}$.  

The first stratification is classical, and is induced from the Newton stratification of the geometric special fiber $\ov{\mathscr{S}}_{K}$ over $\ov{\bF}_p$ of the integral model: 
for each $b\in B(G_{\bQ_p, \mu^{-1}})$, the inverse images of the stratum $\ov{\mathscr{S}}_K^b$ along the specialization maps define locally closed subspaces
\[
\cS_{K, C}^{\circ b} \subset \cS_{K, C}^{\circ}, \quad \cS_{K^p, C}^{\circ b} \subset \cS_{K^p, C}^{\circ}. 
\]
Note that
\[
\cS_{K,C}^{\circ \geq b}:= \bigcup_{b' \geq b} \cS_{K, C}^{\circ b'}
=\cS \setminus\bigcup_{b' \not\geq b} \cS_{K, C}^{\circ b'}
\]
is a quasicompact open subspace of $\cS_{K, C}^{\circ}$ for every $b$. 
In particular, for $\mu$-ordinary $b_0$, the unique maximal element, $\cS_{K, C}^{\circ b_0}$, $\cS_{K^p, C}^{\circ b_0}$ are open strata. 

On the other hand, Caraiani-Scholze defined the Newton stratification of $\Fl_C$ with the reverse direction \cite{CS}*{Section 3}, e.g., $\Fl^{b_0}_C$ is a closed stratum. 
In fact, it is induced from the stratification of $\Bun_{G_{\bQ_p}}$ of the Fargues-Fontaine curve. 
Let us recall this description.
The flag variety $\Fl_C$ as a diamond over $\Spd (C)$ has the following interpretation: it sends a perfectoid space $S$ over $C^{\flat}$ to the set of isomorphism classes of 
modification of the trivial $G_{\bQ_p}$-bundle $\cE_1$ on the relative Fargues-Fontaine curve $X_{S}$ bounded by $\mu$, i.e., pairs
\[
(\cE, i\colon \cE_1 \to \cE)
\]
of a $G_{\bQ_p}$-bundle $\cE$ and a modification $i$ that is bounded by $\mu$ along the divisor determined the structure map $\Fl_C\to \Spd (C)$. 
There is a natural map 
\[
\Fl_C \to \Bun_{G_{\bQ_p}} ; \quad (\cE, i) \to \cE
\]
to the moduli stack $\Bun_{G_{\bQ_p}}$ of $G_{\bQ_p}$-bundles \cite{FS}*{III}. 
The moduli stack $\Bun_{G_{\bQ_p}}$ is an $\ell$-cohomologically smooth Artin v-stack of dimension 0 \cite{FS}*{IV.1.19}.  
As explained in \cite{FS}*{III}, each $b \in B(G_{\bQ_p})$ gives rise to a locally closed subfunctor
\[
\Bun_{G_{\bQ_p}}^b \hookrightarrow \Bun_{G_{\bQ_p}}
\]
classifying $G$-bundles of type $b \in B(G_{\bQ_p})$. In fact, this defines a stratification of $\Bun_{G_{\bQ_p}}$. 
The stratification of $\Fl_C$ is induced from that of $\Bun_{G_{\bQ_p}}$ via the map $\Fl_C\to \Bun_{G_{\bQ_p}}$ as is clear from the definition in \cite{CS}*{3.5.6}. 
In particular, $\Fl^b$ is a locally closed locally spatial subdiamond of $\Fl_C$; it was merely a locally closed topological subset in \cites{CS}. 

The following lemma is useful later:

\begin{lem}[\cite{Hansen:middle}*{2.10}]\label{map to BunG}
Fix $\ell\neq p$. 
The map $\Fl_C \to \Bun_G$ factors through a map
\[
q\colon [\Fl_C / \underline{G(E_{\fp})}] \to \Bun_G, 
\]
and $q$ is $\ell$-cohomologically smooth. 
Moreover, for any open compact subgroup $K_p \subset G(\bQ_p)$, the composite
\[
q_K \colon [\Fl_C / \underline{K_p}] \to [\Fl_C / \underline{G(E_{\fp})}] \xrightarrow{q} \Bun_G
\]
is $\ell$-cohomologically smooth. 
\end{lem}

Fix $b\in B(G_{\bQ_p}, \mu^{-1})$. 
The intersection
\[
\cS_{K^p, C}^{\circ b} \cap (\pi_{HT}^{\circ})^{-1} (\Fl^b_C)=
\cS_{K^p, C}^{\circ b} \times_{\Fl_C} \Fl^b_C
\]
defines a quasicompact and quasiseparated open immersion to $(\pi_{HT}^{\circ})^{-1} (\Fl^b_C)$, and it contains all points of rank 1. 
Let
\[
\pi_{HT}^{\circ b}\colon \cS_{K^p, C}^{\circ b} \times_{\Fl_C} \Fl^b_C
\to \Fl^b_C
\]
denote the map induced from $\pi_{HT}^{\circ}$, which is quasicompact and quasiseparated. 

\begin{lem}\label{rank 1 points}
A natural map
\[
(R\pi_{HT*}^{\circ} \overline{\bF}_{\ell})|_{\Fl_C^b} \to 
R\pi_{HT*}^{\circ b}\overline{\bF}_{\ell}
\]
is an isomorphism in $D_{\et}(\Fl_C^b, \overline{\bF}_{\ell})$ \cite{Scholze:diamond}*{14.13}. 
\end{lem}

\begin{proof}
As all diamonds involved are locally spatial, we are allowed to work with $D^+ (\Fl^b_{C, \et}, \overline{\bF}_{\ell})$. 
It suffices to show that
\[
(R \pi_{HT*}^{\circ}\overline{\bF}_{\ell})_{\ov{y}} \to
(R \pi_{HT*}^{\circ b}\overline{\bF}_{\ell})_{\ov{y}}
\]
is an isomorphism for all geometric points $\overline{y}\colon \Spa (C', (C')^+) \to \Fl_C^b$, because $\Fl^b_{C, \et}$ has enough points \cite{Scholze:diamond}*{14.3}. 
This follows from \cite{CS}*{4.4.1, 4.4.2}.  
\end{proof}

We next recall the product formula that describes the structure of $\pi_{HT}^{\circ b}$. 
We need some preparation. 

Let $\pi_{HT}^b\colon \cM_{(G, b,\mu), \infty} \to \Fl_C$ denote the Hodge-Tate period map from (the diamond associated with) the Rapoport-Zink space at infinite level, which factors through $\Fl^b_C$ \cite{CS}*{4.2.6}. Using the interpretation as the moduli of $G_{\bQ_p}$-shtukas \cite{Scholze:Berkeley}*{24.3.5}, we have the following cartesian diagram of v-stacks
\[
\begin{CD}
\cM_{(G, b,\mu), \infty} @>>> * \\
@V \pi_{HT}^b VV @VVV \\
\Fl^b_C @>>>  \Bun_{G_{\bQ_p}}^b.  
\end{CD}
\]
Let us also recall that $\Bun_{G_{\bQ_p}}^b$ has the form of $[* /\widetilde{J}_b]$, where $\widetilde{J}_b$ is the extension of $\underline{J_b (\bQ_p)}$ by a ``unipotent'' group $\widetilde{J}_b^0$ of dimension $d_b:=\langle 2\rho, \nu_b \rangle$, with a canonical splitting $\underline{J_b (\bQ_p)}\hookrightarrow \widetilde{J}_b$ \cite{FS}*{III}. 

The product formula describes the fiber product
\[
\cS_{K^p, C}^{\circ b} \times_{\Fl_C} \cM_{(G, b,\mu), \infty}
\]
using Igusa varieties.  
Choose a completely slope divisible $p$-divisible group $\bX_b$ over $\ov{\bF}_p$ with $G_{\bQ_p}$-structure corresponding to $b\in B(G_{\bQ_p}, \mu^{-1})$. (It is known to exist.)
This choice determines the (perfect) Igusa variety $\Ig^b$ of level $K$ over $\ov{\bF}_p$ \cite{CS}*{4.3.1}, equipped with the action of $J_b (\bQ_p)$. 
Its dimension is $d_b$. 
As $\Ig^b$ is perfect \cite{CS}*{4.3.5}, it lifts uniquely to a flat $p$-adic formal scheme $\Ig^b_{W(\ov{\bF}_p)}$ over $W(\ov{\bF}_p)$. 

\begin{prop}[\cite{CS}*{4.3.19, 4.3.20}]\label{product formula}
Choosing $(\bX_b)_{\cO_C}\in \cM (G_{\bQ_p}, b, \mu)(C)$, we have an isomorphism of diamonds
\[
\cS_{K^p, C}^{\circ b} \times_{\Fl_C} \cM_{(G_{\bQ_p}, b,\mu), \infty}
\cong 
\cM_{(G_{\bQ_p}, b,\mu), \infty} \times_{C} (\Ig^b_{W(\ov{\bF}_p)})^{\textnormal{ad}}_{C}, 
\]
where $(\Ig^b_{W(\ov{\bF}_p)})^{\textnormal{ad}}_{C}$ denotes the adic generic fiber of $\Ig^b_{W(\ov{\bF}_p)}$ base changed to $C$. 
\end{prop}

\section{Mantovan's formula}
We combine results explained in the previous section to prove a form of Mantovan's formula. 
(For usual Mantovan's formula, see \cites{Mantovan, Hamacher-Kim}.)
Fix $\ell\neq p$ and set $C:= \widehat{\ov{\bQ}}_p$. The Rapoport-Zink space at infinite level $\cM_{(G, b, \mu), \infty}$ will be regarded as a diamond over $\Spd (C)$. 

Recall that the perfect Igusa variety $\Ig^b$ is the perfection of the limit of Igusa varieties at finite level \cite{CS}*{4.3.8}. 
Define $R\Gamma_c (\Ig^b, \ov{\bF}_{\ell} (d_b))$ to be the colimit of compactly supported cohomology of $\ov{\bF}_{\ell} (d_b)$ on Igusa varieties at finite level. 
The Poincar\'e duality at finite level implies that
\[
R\Gamma (\Ig^b, \ov{\bF}_{\ell})^{*} \cong 
R\Gamma_c (\Ig^b, \ov{\bF}_{\ell} (d_b))[2d_b], 
R\Gamma_c (\Ig^b, \ov{\bF}_{\ell}(d_b))^{*} \cong 
R\Gamma (\Ig^b, \ov{\bF}_{\ell})[2d_b]
\]
in the derived category of smooth $\ov{\bF}_{\ell}$-representations of $J_b (\bQ_p)$, where $(-)^{*}$ is the smooth dual and $\ov{\bF}_{\ell} (d_b)$ is regarded as an $J_b (\bQ_p)$-equivariant sheaf via the functoriality of dualizing complexes; 
compare with the argument in \cite{Kret-Shin}*{7.1}. 

\begin{thm}\label{Mantovan}
There is a filtration of $R\Gamma (\cS_{K^p, \ov{\bQ}}, \ov{\bF}_{\ell})$ by complexes of smooth representations of $G(\bQ_p)\times W_{E_\fp}$ whose graded pieces are
\begin{multline*}
R\Gamma (\Ig^b, \ov{\bF}_{\ell})^{\op}\otimes^L_{C_c (J_b (\bQ_p))} R\Gamma_c (\cM_{(G, b, \mu), \infty}, \ov{\bF}_{\ell}(d_b))[2d_b] \cong \\
(R\Gamma_c (\Ig^b, \ov{\bF}_{\ell}(d_b))^{*})^{\op}\otimes^L_{C_c (J_b (\bQ_p))} R\Gamma_c (\cM_{(G, b, \mu), \infty}, \ov{\bF}_{\ell}(d_b)). 
\end{multline*}
Here, $\ov{\bF}_{\ell}(d_b)$ on $\cM_{(G, b, \mu), \infty}$ is equipped with a $J_b (\bQ_p)$-equivariant structure as in Lemma \ref{identification of upper shriek} below. 
An analogous claim holds at the level $K^p K_p$ by taking $K_p$-invariants. 
\end{thm}

\begin{rem}
The dual statement gives a filtration of $R\Gamma_c (\cS_{K^p, \ov{\bQ}}, \ov{\bF}_{\ell}(d))[2d]$ whose graded pieces are
\[
R\Hom_{J_b (\bQ_p)}(R\Gamma_c (\cM_{(G, b, \mu), \infty}, \overline{\bF}_{\ell}(d_b)), R\Gamma_c (\Ig^b, \ov{\bF}_{\ell}(d_b))).
\]
As the twists on the source and target are cancelled out (see Lemma \ref{dualizing complexes} below), the resulting equality in the Grothendieck group of smooth representations of $G(\bQ_p)\times W_{E_{\fp}}$ is the $\ov{\bF}_{\ell}$-version of usual Mantovan's formula. 

It is not clear to the author if this dual statement implies Theorem \ref{Mantovan} itself as $R\Gamma_c (\cM_{(G, b, \mu), \infty}, \overline{\bF}_{\ell}(d_b))$ is not admissible. 
\end{rem}

The rest of this section is devoted to the proof of this theorem. 
By Lemma \ref{cohomology of good reduction locus}, we have isomorphisms
\[
R\Gamma (\cS_{K^p, \overline{\bQ}}, \overline{\bF}_{\ell}) \cong  
R\Gamma (\cS_{K^p, C}^{\circ}, \ov{\bF}_{\ell}) \cong
R\Gamma (\Fl_C, R\pi_{HT*}\ov{\bF}_{\ell}).  
\]
Using excision sequences, we obtain a filtration whose graded pieces are
\[
R\Gamma_c (\Fl_C^b, R\pi_{HT*}\ov{\bF}_{\ell}) \cong 
R\Gamma_c (\Fl_C^b, R\pi_{HT*}^{\circ b}\overline{\bF}_{\ell}), 
\]
where the isomorphism comes from Lemma \ref{rank 1 points}. 
Therefore, it remains to identify $R\Gamma_c (\Fl_C^b, R\pi_{HT*}^{\circ b}\overline{\bF}_{\ell})$. 
Let us factorize $\pi_{HT}^{b}$ as
\[
\cM_{(G, b,\mu),\infty} \xrightarrow{\pi_{\unip}^b} \cM_{(G, b,\mu),\infty} / \widetilde{J}_{b, C}^{0} \xrightarrow{\pi_{\semi}^{b}} \Fl^b_C, 
\]
where $\pi_{\unip}^b$ (resp. $\pi_{\semi}^b$) is a $\widetilde{J}_{b, C}^0$-torsor (resp. $\underline{J_b (\bQ_p)}$-torsor) as v-sheaves. 

As in \cite{Mantovan}*{Proposition 5.12}, 
\begin{multline*}
R\Gamma (\Ig^b, \ov{\bF}_{\ell})^{\op}\otimes^L_{C_c (J_b (\bQ_p))} R\Gamma_c (\cM_{(G, b, \mu), \infty}, \ov{\bF}_{\ell}) \cong \\
(R\Gamma (\Ig^b, \ov{\bF}_{\ell})\otimes^L_{\ov{\bF}_{\ell}} R\Gamma_c (\cM_{(G, b, \mu), \infty}, \ov{\bF}_{\ell}))\otimes^L_{C_c (J_b (\bQ_p))} \ov{\bF}_{\ell}. 
\end{multline*}
So, it suffices to show the following lemmas:

\begin{lem}\label{HS SS}
Choose an $\ov{\bF}_{\ell}$-valued Haar measure of $J_b (\bQ_p)$.  
There is a $G(\bQ_p)$-equivariant isomorphism
\[
R\Gamma_c (\Fl_C^b, R\pi_{HT*}^{\circ b}\overline{\bF}_{\ell}) 
\cong \\
R\Gamma_c (\cM_{(G,b,\mu), \infty}/ \widetilde{J}_{b, C}^{0}, 
R\pi_{\semi}^{b \, *} R\pi_{HT*}^{\circ b}\overline{\bF}_{\ell})
\otimes^L_{C_c (J_b (\bQ_p))} \overline{\bF}_{\ell}.  
\]
\end{lem}

\begin{lem}\label{identification of upper shriek}
The map $\pi_{\unip}^b$ is $\ell$-cohomologically smooth of dimension $d_b$ and
\[
R\pi_{\unip}^{b \, !} \bF_{\ell} \cong \ov{\bF}_{\ell}(d_b)[2d_b];
\]
we regard $\ov{\bF}_{\ell}(d_b)$ as a $J_b (\bQ_p)$-equivariant sheaf via this isomorphism. 

Moreover, the natural transform
\[
R\pi^b_{\unip \, !}R\pi_{\unip}^{b \, !} \to \id
\]
is an equivalence. 
\end{lem}

\begin{lem}\label{Kunneth}
There is an isomorphsim
\begin{multline*}
R\Gamma_c (\cM_{(G,b,\mu), \infty}, \pi_{HT}^{b \, *} R\pi_{HT*}^{\circ b}\overline{\bF}_{\ell}\otimes^L_{\overline{\bF}_{\ell}} R\pi_{\unip}^{b\, !} \overline{\bF}_{\ell}) \\
\cong 
R\Gamma (\Ig^b, \ov{\bF}_{\ell})\otimes^L_{\ov{\bF}_{\ell}} R\Gamma_c (\cM_{(G, b, \mu), \infty}, R\pi_{\unip}^{b\, !} \overline{\bF}_{\ell})
\end{multline*}
compatible with actions of $G(\bQ_p)$, $J_b (\bQ_p), W_{E_{\fp}}$. 
\end{lem}

\begin{proof}[Proof of Lemma \ref{HS SS}]
We first claim that
\begin{multline*}
R\Gamma_c (\cM_{(G,b,\mu), \infty}/ \widetilde{J}_{b, C}^0, 
R\pi_{\semi}^{b \, *} R\pi_{HT*}^{\circ b}\overline{\bF}_{\ell})
\otimes^L_{C_c (J_b (\bQ_p))} \overline{\bF}_{\ell} 
\cong  \\
R\Gamma_c (\Fl^b_C, R\pi_{\semi !}^{b}R\pi_{\semi}^{b \, *} R\pi_{HT*}^{\circ b}\overline{\bF}_{\ell} \otimes^L_{C_c (J_b (\bQ_p))} \overline{\bF}_{\ell} ). 
\end{multline*}
This is essentially \cite{Mantovan}*{Theorem 5.2} applied to $R\pi_{\semi !}^{b}R\pi_{\semi}^{b \, *} R\pi_{HT*}^{\circ b}\overline{\bF}_{\ell}$ on $\Fl^b_C$, or its extension by zero on $\Fl_C$. While the setting is different, the argument there works as we can still take Godement resolutions of $J_b (\bQ_p)$-equivariant sheaves. (The resolution is taken in the categories of not necessarily smooth $J_b (\bQ_p)$-equivariant sheaves.)
So, we shall identify
\[
R\pi_{\semi !}^{b}R\pi_{\semi}^{b \, *} R\pi_{HT*}^{\circ b}\overline{\bF}_{\ell}
\otimes^L_{C_c (J_b (\bQ_p))} \overline{\bF}_{\ell}. 
\]
By the projection formula \cite{Scholze:diamond}*{22.23}, 
\[
R\pi_{\semi !}^{b}R\pi_{\semi}^{b \, *} R\pi_{HT*}^{\circ b}\overline{\bF}_{\ell} \cong 
R\pi_{\semi !}^b \ov{\bF}_{\ell} \otimes^L_{\ov{\bF}_{\ell}} R\pi_{HT*}^{\circ b}\overline{\bF}_{\ell}, 
\]
and every stalk of $R\pi_{\semi !}^b \ov{\bF}_{\ell}$ is isomorphic to $C_c (J_b (\bQ_p))$. 
One has a map, depending on the choice of the Haar measure, 
\[
R\pi_{\semi !}^b \ov{\bF}_{\ell} \otimes^L_{\ov{\bF}_{\ell}} R\pi_{HT*}^{\circ b}\overline{\bF}_{\ell} 
\to R\pi_{HT*}^{\circ b}\overline{\bF}_{\ell}
\]
such that it is given by integration at the level of stalks. Cf.~\cite{Hamacher-Kim}*{B.3} in the setting of schemes. 
It induces an isomorphsim
\[
R\pi_{\semi !}^{b}R\pi_{\semi}^{b \, *} R\pi_{HT*}^{\circ b}\overline{\bF}_{\ell}
\otimes^L_{C_c (J_b (\bQ_p))} \overline{\bF}_{\ell} \cong 
R\pi_{HT*}^{\circ b}\overline{\bF}_{\ell}. 
\]
This completes the proof of Lemma \ref{HS SS}.  
\end{proof}

\begin{proof}[Proof of Lemma \ref{identification of upper shriek}]
As $\pi^b_{\unip}$ is a $\widetilde{J}_b^{0}$-torsor, it follows from \cite{FS}*{III.5.1} that $\pi^b_{\unip}$ is $\ell$-cohomologically smooth of dimension $d_b$.  
The equivalence of 
\[
R\pi^b_{\unip !} R\pi^{b \, !}_{\unip} \to \id
\]
is shown in the proof of \cite{FS}*{V.2.1}. 
It is also shown there that $R\pi^{b\, *}_{\unip}$ is fully faithful. 
To identify $R\pi^{b \, !}_{\unip} \ov{\bF}_{\ell}$, note that $\widetilde{J}^0_{b, W(\ov{\bF}_{p})}$ is the v-sheaf associated with the formal scheme of the form of
\[
\textnormal{Spf} (W(\ov{\bF}_{p})[\![x_1^{1/p^{\infty}}, \dots, x_{d_b}^{1/p^{\infty}}]\!]);
\]
see \cite{CS}*{4.2.11}. 
So, by \cite{Scholze:diamond}*{Section 27}, the dualizing complex of 
\[
\widetilde{J}^0_{b, W(\ov{\bF}_{p})} \to \Spd (W(\ov{\bF}_p))
\]
is naturally isomorphic to $\ov{\bF}_{\ell}(d_b)[2d_b]$. So, the same holds for
\[
\Spd (W(\ov{\bF}_p)) \to [\Spd (W(\ov{\bF}_p))/ \widetilde{J}^0_{b, W(\ov{\bF}_{p})}], 
\]
which in turn gives the desired identification $R\pi^{b \, !}_{\unip} \ov{\bF}_{\ell}\cong \ov{\bF}_{\ell}(d_b)[2d_b]$. 
\end{proof}

\begin{proof}[Proof of Lemma \ref{Kunneth}]
By Proposition \ref{product formula}, we have the following carteisian diagram of locally spatial diamonds:
\[
\begin{CD}
\cM_{(G_{\bQ_p}, b,\mu), \infty} \times_{C} (\Ig^b_{W(\ov{\bF}_p)})^{\textnormal{ad}}_{C} 
@> \pr >> \cM_{(G_{\bQ_p}, b,\mu), \infty} \\
@VVV @VV \pi^b_{HT} V \\
\cS_{K^p, C}^{\circ b} @>> \pi^{\circ b}_{HT} > \Fl_C^b. 
\end{CD}
\]
Using the base change isomorphisms twice, we see that $R\pi_{HT}^{b \, *} R\pi_{HT*}^{\circ b}\overline{\bF}_{\ell}$ identifies with the pullback of $R\Gamma ((\Ig^b_{W(\ov{\bF}_p)})^{\textnormal{ad}}_{C}, \overline{\bF}_{\ell})$ along the structure map 
\[
\cM_{(G_{\bQ_p}, b,\mu), \infty} \to \Spd (C). 
\]
As there is an isomorphism supplied by \cite{CS}*{4.4.3} 
\[
R\Gamma ((\Ig^b_{W(\ov{\bF}_p)})^{\textnormal{ad}}_{C}, \overline{\bF}_{\ell}) \cong
R\Gamma (\Ig^b, \ov{\bF}_{\ell}), 
\]
we finish by the projection formula \cite{Scholze:diamond}*{22.23}. 
\end{proof}

Let us remark that the twist has no effect in the following sense: 

\begin{lem}\label{dualizing complexes}
The $J_b (\bQ_p)$-equivariant structure on $\ov{\bF}_{\ell}(d_b)$ on $\Ig^b$ and on $R\pi_{\unip}^{b \, !} \ov{\bF}_{\ell}$ are given by the same smooth character $\kappa\colon J_b (\bQ_p)\to \ov{\bF}_{\ell}^\times$. 
\end{lem}

\begin{proof}
The first remark is that $\Ig^b$ is nonempty \cite{VW}*{11.2}.  
It is clear that the action of $J_b (\bQ_p)$ on $\ov{\bF}_{\ell}(d_b)_{\Ig^b}$ is given by some character on the $J_b (\bQ_p)$-orbit of a connected component of $\Ig^b$ (which we choose); it will be denoted by $\kappa_1$. 
Moreover, writing $\Ig^b$ as the perfection of the inverse limit of Igusa varieties at finite level, one sees that this character $\kappa_1$ is smooth. 

It remains to study $R\pi_{\unip}^{b \, !} \ov{\bF}_{\ell}$. 
Recall the cartesian diagram
\[
\begin{CD}
\cM_{(G_{\bQ_p},b,\mu),\infty} @>>> \Spd (C) \\
@V \pi_{\unip}^b VV @VVV \\
\cM_{(G_{\bQ_p},b,\mu),\infty} / \widetilde{J}_{b, C}^0 @>>> [ \Spd (C)  / \widetilde{J}^{0}_{b, C}]. 
\end{CD}
\]
By commutation of upper shriek with base change \cite{Scholze:diamond}*{23.12}, $R\pi_{\unip}^{b \, !} \ov{\bF}_{\ell}$ is obtained from the pullback from the corresponding object on $\Spd (C)$, on which $J_b (\bQ_p)$ acts by a character since it is already isomorphic to $\ov{\bF}_{\ell}(d_b)[2d_b]$ on $\Spd (C)$ as in the proof of Lemma \ref{identification of upper shriek}. 
Write $\kappa_2$ for this character. 

To compare $\kappa_1$ and $\kappa_2$, we need only observe that $\widetilde{J}^b_{\ov{\bF}_{p}}$ acts freely on the v-sheaf $\Ig_{W(\ov{\bF}_{p})}^{b, \diamondsuit}$ associated with the formal scheme $\Ig^b_{W(\ov{\bF}_{p})}$.  
(In particular, $\kappa_1$ is independent of the choice of the connected component and $\kappa_2$ is also smooth.)
\end{proof}

\section{Semiperversity}
From now on, we work in the setting of \cite{CS:noncompact}*{Section 2}. 
So, $B=F$ is a CM field and $V=F^{2n}$, and $G$ is a quasi-split similitude unitary group. 
From \cite{CS:noncompact}, we recall the key semiperversity result. 

Let $C$ be a complete algebraically closed nonarchimedean extension of $\bQ_p$ with the ring of integers $\cO_C$ and the residue field $k$. 

\begin{thm}[\cite{CS:noncompact}*{4.6.1}]\label{semiperversity}
There is a cofinal system of formal models $\mathfrak{Fl}$ of $\Fl$ over $\cO_C$ such that the nearby cycle
\[
R\psi (R\pi_{HT*}^{\circ} \bF_{\ell}) 
\]
belongs to $^{p}D^{\geq d}(\Fl_k, \bF_{\ell})$. 
\end{thm}

\begin{cor}\label{costalk}
Let $i^{b_0}$ denote the the closed immersion $\Fl (\bQ_p) \hookrightarrow \Fl_C$. 
The local cohomology
\begin{align*}
R\Gamma_{\Fl (\bQ_p)} (\Fl_C, R\pi_{HT*}^{\circ} \bF_{\ell})
&:=
R\Gamma (\Fl (\bQ_p), Ri_{b_0}^! R\pi_{HT*}^{\circ} \bF_{\ell}) \\
&\cong
R\Gamma_c (\Fl (\bQ_p), Ri_{b_0}^! R\pi_{HT*}^{\circ} \bF_{\ell})
\end{align*}
belongs to $D^{\geq d}(\bF_{\ell})$. 
Similarly, 
\[
\Gamma_c ([\Fl (\bQ_p)/K_p], Ri_{b_0}^! R\pi_{HT*}^{\circ} \bF_{\ell})
\]
belongs to $D^{\geq d}(\bF_{\ell})$. 
\end{cor}

\begin{proof}
Set $\cF:= R\pi_{HT*}^{\circ} \bF_{\ell}$. 
Let $j^{b_0}\colon \Fl_C \setminus \Fl (\bQ_p) \hookrightarrow \Fl$ denote the open immersion from the complement of the ordinary locus. 
We will work in the $\infty$-categorical setup. 
By Lemma \ref{Gaitsgory-Lurie} below, there is a canonical fiber sequence
\[
i_{*}^{b_0}Ri^{b_0!}\cF \to \cF \to Rj^{b_0}_{*}j^{b_0*} \cF
\]
giving rise to a canonical fiber sequence
\[
R\Gamma_{\Fl (\bQ_p)} (\Fl_C, \cF) \to R\Gamma (\Fl_C, \cF) \to R\Gamma (\Fl_C \setminus \Fl (\bQ_p), \cF). 
\]

Let us first fix a formal model $\mathfrak{Fl}$ as in Theorem \ref{semiperversity}. 
The image of the ordinary locus under the specialization map is a closed set $i_Z\colon Z\hookrightarrow \Fl_k$ consisting of finitely many closed points. 
As $i_Z^!$ is left $t$-exact, $i_Z^{!} R\psi (\cF)$ lives in $D^{\geq d}(Z, \bF_{\ell})$.  
Let $j_Z \colon \Fl_k \setminus Z \to \Fl_k$ denote the open immersion from the complement of $Z$. 
As above, there is a canonical fiber sequence 
\[
i_{Z*}Ri_Z^{!}R\psi (\cF) \to R\psi(\cF) \to Rj_{Z*}j_Z^* R\psi (\cF)
\]
giving rise to a canonical fiber sequence
\[
R\Gamma_{Z} (\Fl_k, R\psi (\cF)) \to R\Gamma (\Fl_C, \cF) \to R\Gamma (\Fl_C \setminus\textnormal{sp}^{-1}(Z), \cF)
\]
with $R\Gamma_{Z} (\Fl_k, R\psi (\cF))\in D^{\geq d}(\bF_{\ell})$. 

Now taking the limits by varying $\mathfrak{Fl}$,  we obtain a canonical fiber sequence
\[
R\lim R\Gamma_{Z} (\Fl_k, R\psi (\cF)) \to R\Gamma (\Fl_C, \cF) \to R\Gamma (\Fl_C \setminus \Fl (\bQ_p), \cF) 
\]
as fibre sequences commute with limits.   
So, we conclude that
\[
R\Gamma_{\Fl (\bQ_p)}(\Fl_C, \cF) \cong R\lim R\Gamma_{Z} (\Fl_k, R\psi (\cF)) \in D^{\geq d}(\bF_{\ell})
\]
as desired. 

The second part follows from the first part. 
\end{proof}

The following lemma was used above:

\begin{lem}\label{Gaitsgory-Lurie}
Let $i\colon Z\hookrightarrow X$ be a closed immersion of locally spatial diamonds with the open complement $j\colon U\hookrightarrow X$. 
The functor $i_*$ admits a left adjoint $Ri^!$, and for any object $\cF$ of the $\infty$-category $\cD_{\et}(X, \bF_{\ell})$ \cite{Scholze:diamond}*{17.1}, there is a canonical fiber sequence
\[
i_* Ri^! \cF \to \cF \to 
Rj_* j^* \cF. 
\]
In particular, $Ri^! $ agrees with the one in \cite{Scholze:diamond}. 
A similar statement holds for schemes. 
\end{lem}

\begin{proof}
One sees that the essential image of the fully faithful functor $i_*$ consists of objects $\cG$ with $j^* \cG\cong 0$. 
Thus, the fiber sequence characterizes the left adjoint functor $Ri^!$, and it is well-defined. A more detail is given in \cite{Gaitsgory-Lurie}*{2.2.5.5} for quasi-projective schemes over an algebraically closed field (this is not an essential assumption; see also \cite{Gaitsgory-Lurie}*{2.2.5.6}). 
\end{proof}

\section{Proof of Theorem \ref{CS}}
We only consider $H^*$; the case of $H^*_c$ follows from this case by the Poincar\'e duality as the dual of a generic unramified $L$-parameter is generic unramified. 
(See the proof of \cite{CS:noncompact}*{1.1}.)

By Proposition \ref{Key 1}, we have
\[
R\Gamma (\Ig^b, \overline{\bF}_{\ell})^{\op} \otimes^L_{C_c (J_b (\bQ_p))} R\Gamma_c (\cM_{(G_{\bQ_p}, b, \mu), K_p}, \overline{\bF}_{\ell}(d_b))_{\fm_p}
\cong 
0.
\]
for any non-ordinary $b$. 
So, by Theorem \ref{Mantovan}, we see that
\[
R\Gamma (S_K, \ov{\bF}_{\ell})_{\fm_p} \cong 
R\Gamma (\Ig^{b_0}, \overline{\bF}_{\ell})^{\op} \otimes^L_{C_c (J_{b_0} (\bQ_p))} R\Gamma_c (\cM_{(G_{\bQ_p}, b_0, \mu), K_p}, \overline{\bF}_{\ell}(d))_{\fm_p}[2d], 
\]
where $b_0$ is the unique ordinary element. 
As in the proof of Theorem \ref{Mantovan}, we have
\begin{multline*}
R\Gamma (\Ig^{b_0}, \overline{\bF}_{\ell})^{\op} \otimes^L_{C_c (J_{b_0} (\bQ_p))} R\Gamma_c (\cM_{(G_{\bQ_p}, b_0, \mu), K_p}, \overline{\bF}_{\ell}(d))_{\fm_p}[2d] \\
\cong 
R\Gamma_c ([\Fl (\bQ_p)/ \underline{K_p}], i^{b_0*}R\pi^{\circ}_{HT*}\ov{\bF}_{\ell})_{\fm_p},  
\end{multline*}
where $i^{b_0}\colon \Fl (\bQ_p)\hookrightarrow \Fl_C$ is a closed immersion from the ordinary stratum. 
We know, by Corollary \ref{costalk}, that 
\[
R\Gamma_c ([\Fl (\bQ_p)/ \underline{K_p}], Ri^{b_0 !} R\pi^{\circ}_{HT*}\ov{\bF}_{\ell})_{\fm_p}
\]
belongs to $D^{\geq d}(\ov{\bF}_{\ell})$. 
We shall prove that 
\[
R\Gamma_c ([\Fl (\bQ_p)/\underline{K_p}], Ri^{b_0 !} R\pi^{\circ}_{HT*}\ov{\bF}_{\ell})_{\fm_p}
\cong
R\Gamma_c ([\Fl (\bQ_p)/\underline{K_p}], i^{b_0*} R\pi^{\circ}_{HT*}\ov{\bF}_{\ell})_{\fm_p}. 
\]
This is the content of Proposition \ref{Key 2}. 

\begin{proof}[Proof of Proposition \ref{Key 2}]
It suffices to show that
\[
R\Gamma_c ([\Fl (\bQ_p) /\underline{K_p}], Ri^{b_0 !} i^b_! i^{b*}R\pi^{\circ}_{HT*}\ov{\bF}_{\ell})_{\fm_p} \cong 0
\]
for every $b\neq b_0$. 
From the proof of Theorem \ref{Mantovan}, one sees that $i^{b*}(R\pi^{\circ}_{HT*}\ov{\bF}_{\ell})$ comes from some object
\[
V_b \in D_{\et}(\Bun_{G_{\bQ_p}}^b, \ov{\bF}_{\ell}) \cong D(J_b (\bQ_p), \ov{\bF}_{\ell})
\]
under the pullback along $q^b_{K_p}\colon [\Fl^b_C /K_p] \to \Bun^b_{G_{\bQ_p}}$ (the restriction of $q_{K_p}$), where the equivalence is given by \cite{FS}*{V.2.2}. 
Moreover, $V_b$ is a bounded complex of admissible representations of $J_b (\bQ_p)$. 
As $q_{K_p}$ is $\ell$-cohomologically smooth by Lemma \ref{map to BunG}, there is an identification
\[
Ri^{b_0 !} i^b_! i^{b*}R\pi^{\circ}_{HT*}\ov{\bF}_{\ell}\cong
q^{b_0 *}_{K_p} (Ri^{b_0 !}i^b_! V_b) 
\]
given by smooth base change \cite{Scholze:diamond}*{23.16.(iii)}.  
Namely, $Ri^{b_0 !} i^b_! i^{b*}R\pi^{\circ}_{HT*}\ov{\bF}_{\ell}$ comes from an object of $D(J_{b_0} (\bQ_p), \ov{\bF}_{\ell})$ by regarding $Ri^{b_0 !}i^b_! V_b$ as such an object.  
As in the proof of Theorem \ref{Mantovan}, we see that $R\Gamma_c ([\Fl (\bQ_p) /\underline{K_p}], Ri^{b_0 !} i^b_! i^{b*}R\pi^{\circ}_{HT*}\ov{\bF}_{\ell})_{\fm_p}$ is given by
\[
(Ri^{b_0 !}i^b_! V_b)^{\op} \otimes^L_{C_c (J_{b_0}(\bQ_p))} R\Gamma_c (\cM_{(G_{\bQ_p}, b_0, \mu), K_p}, \ov{\bF}_{\ell}(d))_{\fm_p}[2d].   
\]
It suffices to show the vanishing
\[
R\Hom_{J_{b_0}(\bQ_p)} (R\Gamma_c (\cM_{(G_{\bQ_p}, b_0, \mu), K_p}, \ov{\bF}_{\ell}(d))_{\fm_p}, (Ri^{b_0 !}i^b_! V_b)^{*})\cong 0. 
\]
We need the following lemma to proceed:

\begin{lem}
The Artin v-stack $\Bun_{G_{\bQ_p}}^{b_0}$ is $\ell$-cohomologically smooth of dimension $-d$. 
Moreover, the dualizing complex $D_{\Bun_{G_{\bQ_p}}^{b_0}}$ on $\Bun_{G_{\bQ_p}}^{b_0}$ has the form of
\[
\kappa^{-1} [-2d] \in D (J_{b_0}(\bQ_p), \ov{\bF}_{\ell}) \cong D_{\et} (\Bun_{G_{\bQ_p}}^{b_0}, \ov{\bF}_{\ell})
\]
for the character $\kappa\colon J_{b_0}(\bQ_p)\to \ov{\bF}_{\ell}^{\times}$ in Lemma \ref{dualizing complexes}. 
\end{lem}

\begin{proof}
The first part is \cite{FS}*{IV.1.22}. For the second part, let us consider maps
\[
[* / \underline{J_{b_0} (\bQ_p)}] \xrightarrow{s}
\Bun_{G_{\bQ_p}}^{b_0}\cong [* / \widetilde{J}_{b_0}] \xrightarrow{p_{\unip}}
[* / \underline{J_{b_0} (\bQ_p)}] \xrightarrow{p_{\semi}} *, 
\]
where $p_{\unip}, p_{\semi}$ are natural projections and $s$ is a section of $p_{\unip}$. 
As in the proof of Lemma \ref{dualizing complexes}, $Rs^! \ov{\bF}_{\ell}$ identifies with $\kappa [2d]$ and 
$R s_! \kappa [2d] \cong \ov{\bF}_{\ell}$. 
Since $s$ is a section of $p_{\unip}$, we see that 
\[
Rp_{\unip}^! \kappa[2d] \cong \ov{\bF}_{\ell}, \quad
Rp_{\unip}^! \ov{\bF}_{\ell} \cong \kappa^{-1}[-2d]. 
\]
So, it suffices to show that $Rp_{\semi}^! \ov{\bF}_{\ell}\cong \ov{\bF}_{\ell}$. 
For this, we show that the dualizing complex $D_{J_b^{\diamondsuit} / \underline{J_b (\bQ_p)}}$ on $J_b^{\diamondsuit} / \underline{J_b (\bQ_p)}$ identifies with 
$\bF_{\ell}[2 \dim J_b]$. 
Consider a tower $\{ J_b^{\diamondsuit} / \underline{J}\}_{J}$ \'etale over $J_b^{\diamondsuit} / \underline{J_b (\bQ_p)}$ for open compact pro-$p$ subgroups $J \subset J_b (\bQ_p)$, with Hecke action. 
Fixing an $\ov{\bF}_{\ell}$-valued Haar measure of $J_b (\bQ_p)$, we can identify dualizing complexes $D_{J_b^{\diamondsuit} /\underline{J}}$ with 
$\ov{\bF}_{\ell}[2\dim J_b]$ via the pullback to $D(J_b^{\diamondsuit}, \ov{\bF}_{\ell})$ by \cite{Scholze:diamond}*{24.2}. 
The resulting $J_b (\bQ_p)$-equivarinat structure on $\ov{\bF}_{\ell}[2\dim J_b]$ is trivial, so this gives the desired identification of $D_{J_b^{\diamondsuit} / \underline{J_b (\bQ_p)}}$ as $D_{J_b^{\diamondsuit} /\underline{J}}$ are pullbacks of $D_{J_b^{\diamondsuit} /\underline{J_b (\bQ_p)}}$. 
\end{proof}

By this lemma, we have 
\begin{align*}
&R\Hom_{J_{b_0}(\bQ_p)} (R\Gamma_c (\cM_{(G_{\bQ_p}, b_0, \mu), K_p}, \ov{\bF}_{\ell}(d))_{\fm_p}, (Ri^{b_0 !}i^b_! V_b)^{*})[-2d] \\
&\cong
R\Hom_{J_{b_0}(\bQ_p)} (R\Gamma_c (\cM_{(G_{\bQ_p}, b_0, \mu), K_p}, \ov{\bF}_{\ell})_{\fm_p}, \bD(Ri^{b_0 !}i^b_! V_b)) \\
&\cong
R\Hom_{J_{b_0}(\bQ_p)} (R\Gamma_c (\cM_{(G_{\bQ_p}, b_0, \mu), K_p}, \ov{\bF}_{\ell})_{\fm_p}, \bD(Ri^{b_0 !}\bD (\bD (i^b_! V_b)))) \\
&\cong
R\Hom_{J_{b_0}(\bQ_p)} (R\Gamma_c (\cM_{(G_{\bQ_p}, b_0, \mu), K_p}, \ov{\bF}_{\ell})_{\fm_p}, \bD (\bD (i^{b_0 *}\bD (i^{b}_! V_b)))) \\
&\cong
R\Hom_{J_{b_0}(\bQ_p)} (R\Gamma_c (\cM_{(G_{\bQ_p}, b_0, \mu), K_p}, \ov{\bF}_{\ell})_{\fm_p}, i^{b_0 *}\bD (i^{b}_! V_b)), 
\end{align*}
where the second and fourth isomorphisms can be seen from the characterization of reflexible objects \cite{FS}*{V.6.2}. 
To show that this vanishes, we may assume $V_b$ is concentrated in one degree.  
We may also assume that $R\Gamma_c (\cM_{(G_{\bQ_p}, b_0, \mu), K_p}, \ov{\bF}_{\ell})_{\fm_p}$ itself does not vanish, otherwise there is nothing to prove.

We now look at excursion operators. 
The support of the action of the spectral Bernstein center $\cZ^{\spec} (G_{\bQ_p}, \ov{\bF}_{\ell})$ on $R\Gamma_c (\cM_{(G_{\bQ_p}, b_0, \mu), K_p}, \ov{\bF}_{\ell})_{\fm_p}$ has the only closed point corresponding to $\rho_{\fm_p}$ as in the proof of Theorem \ref{local vanishing}, and this $L$-parameter is generic by assumption. 
On the other hand, the action of $\cZ^{\spec} (G_{\bQ_p}, \ov{\bF}_{\ell})$ on $\bD (i^{b}_! V_b)$ is the same as the Verdier dual of the action on $i^b_! V_b$ twisted by the involution $D^{\spec}$ of $\cZ^{\spec} (G_{\bQ_p}, \ov{\bF}_{\ell})$ as follows from \cite{FS}*{IX.2.2}; compare with \cite{FS}*{IX.5.3}. 
In particular, the support of the action on $\bD (i^{b}_! V_b)$ is contained in the support of the action on $i^{b}_! V_b$, up to the involution $D^{\spec}$. 
As $J_b$ is not quasi-split,  the support of the action on $i^{b}_! V_b$ does not contain any point corresponding to a generic unramified $L$-parameter by Lemma \ref{non-quasi-split implies non-generic}. 
As the class of generic unramified $L$-parameters is preserved under $D^{\spec}$, we conclude the same for $\bD (i^{b}_! V_b)$.  
This implies the desired vanishing: assume that $\Ext^i$ is nonzero for some $i$. We may assume $i=0$ after replacing $V_b$. 
Two natural actions of $\cZ^{\spec} (G_{\bQ_p}, \ov{\bF}_{\ell})$, via the target or the source, on 
\[
\Hom_{J_{b_0}(\bQ_p)} (R\Gamma_c (\cM_{(G_{\bQ_p}, b_0, \mu), K_p}, \ov{\bF}_{\ell})_{\fm_p}, i^{b_0 *}\bD (i^{b}_! V_b)) \neq 0
\]
are the same since the action of $\cZ^{\spec} (G_{\bQ_p}, \ov{\bF}_{\ell})$ on representations of $J_{b_0} (\bQ_p)$ is functorial. 
But, this contradicts properties of supports explained above. 
\end{proof}

Finally let us remark that Theorem \ref{compact} can be proved in the same way using \cite{CS} instead:

\begin{proof}[Proof of Theorem \ref{compact}]
In this case, Shimura varieties are compact. 
Using the perversity \cite{CS}*{6.1.3}, we can argue as in the proof of Theorem \ref{CS} to show that 
$H^i (S_{K, \ov{\bQ}}, \ov{\bF}_{\ell})_{\fm_p}\neq 0$ implies $i \geq d$. 
Using the Poincar\'e duality, we also see that $i\leq d$ holds, i.e., $i=d$. 
\end{proof}

\end{document}